\documentclass[11pt,english,a4paper]{article}
\usepackage[centertags]{amsmath}
\usepackage{amssymb}
\usepackage{amsthm}
\usepackage{makeidx}

\newtheorem{theorem}{Theorem}[section]

\newtheorem{lemma}[theorem]{Lemma}
\newtheorem{proposition}[theorem]{Proposition}
\newtheorem{definition}[theorem]{Definition}

\numberwithin{equation}{section} 

\usepackage{graphicx}
\usepackage{amsmath}
\usepackage{amssymb}
\usepackage[table]{xcolor}
\usepackage{amsthm}

\def\bkR{{\rm I\kern-.17em R}}
\def\bkC{{\rm \kern.24em
       \vrule width.05em height1.4ex depth-.05ex
       \kern-.26em C}}
\begin{document}

\renewcommand{\arraystretch}{1.8}

\title{On a $2$-orthogonal polynomial sequence via quadratic decomposition
}

\author{ Teresa Augusta Mesquita
}

\author{T. Augusta Mesquita\footnote{Corresponding author (teresa.mesquita@fc.up.pt; taugusta.mesquita@gmail.com)} }
\date{}
\maketitle
\begin{center}
{\scriptsize
Escola Superior de Tecnologia e Gest\~ao, Instituto Polit{\'e}cnico de Viana do Castelo, Rua Escola Industrial e Comercial de Nun' \'Alvares, 4900-347, Viana do Castelo, Portugal, \&\\
Centro de Matem\'{a}tica da Universidade do Porto, Rua do Campo Alegre, 687, 4169-007 Porto, Portugal\\
}
\end{center}

\begin{abstract}
We apply a symbolic approach of the general quadratic decomposition of polynomial sequences - presented in a previous article referenced herein - to polynomial sequences fulfilling specific orthogonal conditions towards two given functionals $u_{0},u_{1}$ belonging to the dual of the vector space of polynomials with coefficients in $\mathbb{C}$.  The general quadratic decomposition produces four new sets of polynomials whose properties are investigated with the help of the mentioned symbolic approach together with further commands which inquire relevant features, as for instance, the classical character of a polynomial sequence. The computational results are detailed for a wide range of choices of parameters and co-recursive type polynomial sequences are also explored.

\end{abstract}
{\footnotesize
\noindent \textbf{Keywords}: quadratic decomposition, $d$-orthogonal polynomials, $d$-symmetric polynomials,  symbolic computations.\\
\noindent  \textbf{2010 Mathematics Subject Classification} : 42C05 , 33C45  , 68W30 , 33-04}

\section{Introduction}
\label{intro}
The symmetry feature of a polynomial sequence $\{ W_{n}(x) \}_{n \geq 0}$ can be translated by the fact that every element of odd degree $W_{2n+1}(x)$ is an odd polynomial function and every element of even degree $W_{2n}(x)$ is an even  polynomial function. This description is also equivalent to the 
splitting of a symmetric polynomial sequence into two new polynomial sequences $\{P_{n}(x) \}_{n \geq 0}$ and $\{ R_{n}(x) \}_{n \geq 0}$, such that
\begin{eqnarray}
\label{symQD-1} W_{2n}(x)= P_{n}\left( x^2\right),\\
\label{symQD-2} W_{2n+1}(x)= x R_{n}\left( x^2\right).
\end{eqnarray}
This is called the quadratic decomposition (QD) of the symmetric polynomial sequence $\{ W_{n}(x) \}_{n \geq 0}$ and the most famous example is given by the Hermite orthogonal polynomial sequence which QD has the Laguerre orthogonal polynomial sequence both in the role of  $\{P_{n}(x) \}_{n \geq 0}$ and $\{ R_{n}(x) \}_{n \geq 0}$ \cite{x,C:78,GS:1939}. 
This simple QD has been generalised initially for all polynomial sequences (non-symmetric included)  \cite{P:90,Quad2-93} and later allowing the replacement both of $x^2$ by any quadratic mapping $x^2+px+q$ and the monomial $x$ in \eqref{symQD-2} by a generic monic polynomial of degree of one $x-a$ \cite{A:2010}. In particular, it is worth mentioning a widely known QD defined as in \eqref{symQD-1}-\eqref{symQD-2},  though with the quadratic transformation $x^2-1$, of the normalised Gegenbauer orthogonal polynomial sequence that has two Jacobi orthogonal polynomial sequences as  $\{P_{n}(x) \}_{n \geq 0}$ and $\{ R_{n}(x) \}_{n \geq 0}$ \cite{KS:1994}.
In brief, a QD of a symmetric orthogonal polynomial sequence provides two new orthogonal sequences, not necessarily symmetric (e.g. \cite[p.34]{P:90}), and several further studies have showed the potencial of the quadratic decomposition in the study of symmetric and non-symmetric orthogonal polynomial sequences (e.g. \cite{AFL:2008,PM-Tounsi:2012}).

In this paper, we decompose quadratically $2$-orthogonal polynomial sequences, following the wider set up of the QD and we put in evidence a family for which the application of the symbolic implementation of the general QD presented in \cite{ATZ:2018} together with broad symbolic procedures introduced in \cite{T:2012} provides interesting results. In particular, that $2$-orthogonal family yields three further $2$-orthogonal polynomial sequences through the most general quadratic decomposition.

This approach has been applied to some $2$-orthogonal polynomial sequences by means of the most simple cubic decomposition, as we can read for instance in \cite{Douak2classiques}. A generalised notion of symmetry of order $d$ is equivalent to a natural polynomial decomposition of order $d+1$ which defines $d+1$ new polynomial sequences. 
In other words, the most evident decomposition for a $2$-orthogonal sequence and $2$-symmetric would be the cubic decomposition and this technique allowed Douak and Maroni to list already a considerable amount of classical $2$-orthogonal polynomial sequences in \cite{Douak2classiques}.

This article is organised as follows. In Section 2 we recall some notation and basic concepts about polynomial and dual sequences, $d$-symmetry and $d$-orthogonality, for any positive integer $d$. Section 3 is devoted to the  general QD, firstly providing its definition and characterisation as given in  \cite{A:2010}, and secondly presenting the steps followed in the symbolic implementation applied throughout our study. In Section 4 we describe the particular family chosen as one of the main subjects of this paper and the correspondent output of that symbolic implementation. We also show in Section 5 the potential of the computational support in the analysis of the changes produced in the QD by the modification of a finite amount of initial parameters of the given $2$-orthogonal polynomial sequence.
In the final section we prove some recurrence relations fulfilled  by the four sets of polynomials.

\section{Notation and basic concepts}
Let $\mathcal{P}$ be the vector space of polynomials with coefficients in $\bkC$ and let
$\mathcal{P}^{\prime}$ be its topological dual space. We denote by $\langle w ,p\rangle$ the action of the form or linear functional $w \in\mathcal{P}^{\prime }$ on $p\in\mathcal{P}$. In particular,
$\langle w, x^{n}\rangle:=\nolinebreak\left( w \right) _{n},n\geq 0$ represent the moments of $w$.
In the following, we will call polynomial sequence (PS) to any sequence
${\{W_{n}\}}_{n \geq 0}$ such that $\deg W_{n}= n,\; \forall n \geq 0$.
We will also call monic polynomial sequence (MPS) to a PS so that all polynomials have leading coefficient equal to one.

If ${\{W_{n}\}}_{n \geq 0}$ is a MPS, there exists a unique sequence
$\{w_n\}_{n\geq 0}$, $w_n\in\mathcal{P}^{\prime}$, called the dual sequence of $\{W_{n}\}_{n\geq 0}$, such that,
\begin{equation}\label{SucDual}
<w_{n},W_{m}>=\delta_{n,m}\ , \ n,m\ge 0.
\end{equation}

%
\noindent On the other hand, given a MPS ${\{W_{n}\}}_{n \geq 0}$, the expansion of $xW_{n+1}(x)$, defines sequences in $\bkC$, ${\{\beta_{n}\}}_{n \geq 0}$ and
$\{\chi_{n,\nu}\}_{0 \leq \nu \leq n,\; n \geq 0},$ such that
\begin{eqnarray}
&&W_{0}(x)=1, \;\; W_{1}(x)=x-\beta_{0},\label{divisao_ci}\\
 && xW_{n+1}(x)= W_{n+2}(x)+ \beta_{n+1}W_{n+1}(x)+\sum_{\nu=0}^{n}\chi_{n,\nu}W_{\nu}(x).\label{divisao}
\end{eqnarray}
This relation is usually called the structure relation of  ${\{W_{n}\}}_{n \geq 0}$, and ${\{\beta_{n}\}}_{n \geq 0}$ and
$\{\chi_{n,\nu}\}_{0 \leq \nu \leq n,\; n \geq 0}$ are called the structure coefficients (SCs).
\medskip

\begin{definition}\cite{Douak2classiques}
A PS ${\{W_{n}\}}_{n \geq 0}$ is $d$-symmetric if it fulfils
$$W_{n}(\xi_{k}x)=\xi_{k}^{n}W_{n}(x),\;\; n \geq 0, \;\; k=1,2,\ldots, d,$$
where $\xi_{k}=\exp\left( \frac{2ik\pi}{d+1}\right),\;\;
k=1,\ldots,d, \,\; \xi_{k}^{d+1}=1$.

If $d=1$, then $\xi_{1}=-1$ and we meet the definition of a symmetric PS in which we have the following property $W_{n}(-x)=(-1)^{n}W_{n}(x),\;\; n \geq 0.$
\end{definition}

\begin{lemma}\cite{Douak2classiques}\label{lemanaturalCD2-sym}
A PS ${\{W_{n}\}}_{n \geq 0}$ is $d$-symmetric if and only if it fulfils
$$W_{m}(x)=x^{\mu}\displaystyle \sum_{k=0}^{n}   a_{m, (d+1)p+\mu}x^{(d+1)p},$$
\par where $m=(d+1)n+\mu,\;\;
0 \leq \mu \leq d,\; n \geq 0$.
\end{lemma}

\begin{lemma}\cite{variations}\label{lema1}
For each $u \in \mathcal{P}^{\prime}$ and each $m \geq 1$, the two following propositions are equivalent.
\begin{description}
\item[a)]  $\langle u, W_{m-1} \rangle \neq 0,\:\: \langle u, W_{n} \rangle=0,\: n \geq m$.
\item[b)] $\exists \lambda_{\nu} \in \mathbb{C},\:\: 0 \leq \nu \leq m-1,\:\: \lambda_{m-1}\neq 0$ such that
$u=\sum_{\nu=0}^{m-1}\lambda_{\nu} w_{\nu}$. \end{description}
\end{lemma}
\begin{definition}\cite{Douak2classiques,Maroni_toulouse}
Given $\Gamma^{1}, \Gamma^{2},\ldots, \Gamma^{d} \in \mathcal{P}^{\prime}$, $d \geq 1$, the polynomial sequence ${\{W_{n}\}}_{n \geq 0}$ is called d-orthogonal polynomial sequence (d-OPS) with respect to  $\Gamma=(\Gamma^{1},\ldots, \Gamma^{d} )$ if it fulfils 
\begin{equation}\label{d ortogonal}
 \langle \Gamma^{\alpha}, W_{m}W_{n}  \rangle=0, \;\; n \geq md +\alpha,\,\; m \geq 0,
\end{equation}
 \begin{equation}\label{d ortogonal regular}
 \langle \Gamma^{\alpha}, W_{m}W_{md+\alpha-1}  \rangle \neq 0, \;\; m \geq 0,
\end{equation}
for each integer $\alpha=1,\ldots, d$.

The conditions (\ref{d ortogonal}) are called the d-orthogonality conditions and the conditions (\ref{d ortogonal regular}) are called the regularity conditions. In this case, the functional $\Gamma$, of dimension $d$, is said regular.
\end{definition}
The  $d$-dimensional functional $\Gamma$ is not unique. Nevertheless, from Lemma \ref{lema1}, we have:
$$\Gamma^{\alpha}=\sum_{\nu=0}^{\alpha-1}\lambda_{\nu}^{\alpha} u_{\nu}, \;\; \lambda_{\alpha-1}^{\alpha} \neq 0, \; 1\leq \alpha \leq d.$$
Therefore, since $w=(w_{0},\ldots, w_{d-1} )$ is unique, we use to consider the canonical functional of dimension $d$, $w=(w_{0},\ldots, w_{d-1} )$, saying that
${\{W_{n}\}}_{n \geq 0}$ is d-OPS ($d\geq 1$) with respect to
$w=(w_{0},\ldots, w_{d-1})$ if
$$ \langle w_{\nu}, W_{m}W_{n}  \rangle=0, \;\; n \geq md +\nu+1,\,\; m \geq 0,$$
$$  \langle w_{\nu}, W_{m}W_{md+\nu}  \rangle \neq 0, \;\; m \geq 0,$$
for each integer $\nu=0,1,\ldots, d-1$.
It is important to remark that when $d=1$ we meet again the notion of regular orthogonality. Furthermore, the $d$-orthogonality corresponds to the generalisation of the well-known recurrence relation fulfilled by the orthogonal polynomials, as the next result recalls.


\begin{theorem}\cite{Maroni_toulouse} \label{recurrence relation for d ortho}
Let ${\{W_{n}\}}_{n \geq 0}$ be a MPS. The following assertions are equivalent:
\begin{description}
\item[a)]  ${\{W_{n}\}}_{n \geq 0}$ is $d$-orthogonal with respect to $w=(w_{0},\ldots, w_{d-1})$.
\item[b)] ${\{W_{n}\}}_{n \geq 0}$ satisfies a $(d+1)$-order recurrence relation ($d \geq 1$):
$$W_{m+d+1}(x)=(x-\beta_{m+d})W_{m+d}(x)-\sum_{\nu=0}^{d-1}\gamma_{m+d-\nu}^{d-1-\nu}W_{m+d-1-\nu}(x), \:\: m \geq 0,$$
with initial conditions
$$W_{0}(x)=1,\:\: W_{1}(x)=x-\beta_{0}\;\: \textrm{and if }d \geq 2:$$
$$W_{n}(x)=(x-\beta_{n-1})W_{n-1}(x)-\sum_{\nu=0}^{n-2}\gamma_{n-1-\nu}^{d-1-\nu}W_{n-2-\nu}(x), \:\: 2 \leq n \leq  d,$$
and regularity conditions: $\gamma_{m+1}^{0} \neq 0,\: \; m \geq 0.$
\end{description}
\end{theorem}

In other words, when the MPS $\{ W_{n} \}_{n \geq 0}$ is $d$-orthogonal, most of the structure coefficients $\{\chi_{n,\nu}\}_{0 \leq \nu \leq n,\; n \geq 0}$ vanish and the remaining $d$ sequences are renamed as $\gamma$'s. For example, while reading this latest theorem with $d=1$ we encounter the well known recurrence relation fulfilled by an orthogonal MPS (MOPS)
$$W_{m+2}(x)=(x-\beta_{m+1})W_{m+1}(x)-\gamma_{m+1}W_{m}(x)\, ,$$
$$W_{0}(x)=1,\:\: W_{1}(x)=x-\beta_{0}, \quad \gamma_{m+1} \neq 0,\: \; m \geq 0.$$

The classical orthogonal polynomials fulfil many important properties (e.g. \cite{C:78}) and they are characterized by the Hahn's property, that is to say, the monic polynomial sequence obtained through the standard differentiation $D$,  $\{ (n+1)^{-1}DW_{n+1}(x)\}_{n \geq 0}$, is also orthogonal. 
The $d$-orthogonal MPSs fulfilling this property are also called $d$-classical MPSs in the Hahn's sense and there are already several studies  regarding these sequences \cite{Ben cheik & Douak Acad.Sci.Paris,Douak-Appell-d-ortho,Douak-two-Laguerre,Douak2classiques,d-classical,1997-KD-PM-1}  that are still not known comprehensively.

Generally speaking, when we define the MPS of derivatives 
$$W^{[1]}_{n}(x)=(n+1)^{-1}DW_{n+1}(x), \; n \geq 0\, ,$$ of any given MPS $\{ W_{n} \}_{n \geq 0}$  with structure relation \eqref{divisao_ci}-\eqref{divisao},  we easily get by differentiation the following recurrence relation which is a helpful tool in the forthcoming computational definition of $\{ W^{[1]}_{n} \}_{n \geq 0}$:
\begin{align}
\nonumber & W^{[1]}_{n}(x)=\frac{W_{n}(x)}{n+1}+\frac{n}{n+1}\left( x-\beta_{n}  \right)W^{[1]}_{n-1}(x)-\frac{1}{n+1}\sum_{\nu=1}^{n-1}\nu \, \chi_{n-1,\nu}W^{[1]}_{\nu-1}(x)\, ,\\
\label{derivatives-rec}  & n\geq 1\; , \;  W^{[1]}_{0}(x)=1 \, , \;  W^{[1]}_{1}(x)=x-\frac{1}{2}\left(\beta_{0}+\beta_{1}\right).
\end{align}

\medskip
Moreover, we already know that any MOPS may define other MOPSs by a simple change of the initial structure coefficients.
This may be achieved either altering the structure coefficient (SC) $\beta_{0}$ in which case we define a co-recursive sequence, or perturbing the initial sets of SCs $\{\beta_{0},\ldots, \beta_{r}\}$  and $\{\gamma_{1},\ldots, \gamma_{r}\}$ for any positive integer $r$, as we read in the following definition.
\begin{definition}\cite{P:91}\label{defn r-perturbed}
Let ${\{W_{n}\}}_{n \geq 0}$ be an orthogonal MPS satisfying b) of Theorem \ref{recurrence relation for d ortho} for $d=1$ and let $r \geq 1$.
\newline
Given $\mu_{0} \in \bkC$,  $\mu=(\mu_{1},\ldots,  \mu_{r})\: \in \bkC^{r}$ and $\lambda=(\lambda_{1},\ldots,  \lambda_{r})\:
\in \Big(\bkC\setminus \{0\}\Big)^{r}$, where either $ \mu_{r} \neq 0$ or $ \lambda_{r} \neq 1$,
the orthogonal MPS ${\{\tilde{W}_{n}(x)\}}_{n \geq 0}$ defined by
\begin{equation}\label{}
\left\{\begin{array}{l}
\tilde{W}_{0}(x)=1,\quad\tilde{W}_{1}(x)=x-\tilde{\beta}_{0}, \\
\tilde{W}_{n+2}(x)=(x-\tilde{\beta}_{n+1})\tilde{W}_{n+1}(x)-\tilde{\gamma}_{n+1}\tilde{W}_{n}(x)\ ,\quad n \geq 0\ ,
\end{array}
\right.
\end{equation}
with
$$\begin{array}{ll}
\tilde{\beta}_{0}=\beta_{0}+\mu_{0} &,\quad \\
\tilde{\beta}_{n}=\beta_{n}+\mu_{n} &,\quad \tilde{\gamma}_{n}=\lambda_{n}\gamma_{n} \ ,\quad 1\leq n \leq r\ ,\\
\tilde{\beta}_{n}=\beta_{n} & , \quad  \tilde{\gamma}_{n}=\gamma_{n} \ \ \ \ ,\quad n \geq r+1\ ,
\end{array}$$
is called a perturbed sequence of order $r$ of ${\{W_{n}\}}_{n \geq 0}$ and is denoted by
$${\left\{W_{n}\left(\mu_{0};  \begin{array}{c}
\mu\\
\lambda\\
\end{array}; r ; x
 \right)\right\}}_{n \geq 0}.$$
\end{definition}
When a single change on the coefficient $\beta_{0}$ is taken and thus  $ \mu_{n} = 0$ and $ \lambda_{n} = 1\;$ for $1\leq n \leq r\;$ the resultant MOPS is called a co-recursive sequence of ${\{W_{n}\}}_{n \geq 0}$ and is notated by ${\{W_{n}(\mu_{0}; x)\}}_{n \geq 0}$.

A perturbation on a specific set of the initial SCs can always be considered for any $d$-orthogonal MPS, through its recurrence relation, if we assure the regularity conditions  of the perturbed $d$-orthogonal MPS ($\tilde{\gamma}^{0}_{n}\neq 0$). In this work, we will describe characteristics obtained with the help of symbolic manipulations  for a given $2$-orthogonal MPS ${\{W_{n}\}}_{n \geq 0}$ and regarding the $2$-orthogonal MPS ${\{\tilde{W}_{n}\}}_{n \geq 0}$ defined initially as a co-recursive and secondly by means of further perturbations. 
In brief, the recurrence relations fulfilled by each one are as follows, with coefficients $\beta_{n}, \, \gamma_{n+1}^{1}, \, \gamma_{n+1}^{0}, \; n \geq 0$, with $\gamma_{n+1}^{0} \neq 0,\:\:\: n \geq 0$, in view of Theorem \ref{recurrence relation for d ortho}.
\begin{align}
\label{RR2ortho}W_{n+1}(x)=(x-\beta_{n})W_{n}(x)-\gamma_{n}^{1}W_{n-1}(x)-\gamma_{n-1}^{0}W_{n-2}(x),\:\: n \geq 2,\\
\nonumber W_{0}(x)=1,\:\: W_{1}(x)=x-\beta_{0},\:\:
W_{2}(x)=(x-\beta_{1})W_{1}(x)-\gamma_{1}^{1}.
\end{align}
It is often used a lighter notation for the $2$-orthogonality by renaming the gammas as follows (cf. \cite{Douak-two-Laguerre})
\begin{align}
\label{RR2ortho}W_{n+1}(x)=(x-\beta_{n})W_{n}(x)-\alpha_{n}W_{n-1}(x)-\gamma_{n-1}W_{n-2}(x),\:\: n \geq 2,\\
\nonumber W_{0}(x)=1,\:\: W_{1}(x)=x-\beta_{0},\:\:
W_{2}(x)=(x-\beta_{1})W_{1}(x)-\alpha_{1}.
\end{align}
The co-recursive $2$-orthogonal MPS ${\{\tilde{W}_{n}(x)\}}_{n \geq 0}$ or ${\{W_{n}(\mu; x)\}}_{n \geq 0}$ that later on we will consider is defined by the next recurrence.

\begin{align}
\label{RR2ortho-tilde}\tilde{W}_{n+1}(x)=(x-\beta_{n})\tilde{W}_{n}(x)-\alpha_{n}\tilde{W}_{n-1}(x)-\gamma_{n-1}\tilde{W}_{n-2}(x),\:\: n \geq 2,\\
\nonumber \tilde{W}_{0}(x)=1,\:\: \tilde{W}_{1}(x)=x-\beta_{0}-\mu,\:\: \tilde{W}_{2}(x)=(x-\beta_{1})\tilde{W}_{1}(x)-\alpha_{1}.
\end{align}
Let us remark that in the next sections we will notate $\tilde{\beta}_{0}= \tau:= \beta_{0}+\mu$.

\section{Quadratic Decomposition and symbolic implementation}
\noindent
Given a MPS $\{W_n\}_{n\geq 0}$, and fixing a monic quadratic polynomial mapping
\begin{equation}
\:\omega(x)=x^2+px+q\;,\:\:p,q\in \bkC
\end{equation}
and a constant $a\in \bkC,$ it is always possible to associate with it, two MPSs  $\:\left\{P_n\right\}_{n\geq 0}\:$ and $\:\left\{R_n\right\}_{n\geq 0}\:$ 
and two other sequences of polynomials  $\:\left\{a_n\right\}_{n\geq 0}\:$
and $\:\left\{b_n\right\}_{n\geq 0}\:$ through the unique decomposition
\begin{eqnarray}
&& W_{2n}(x) = P_n(\omega (x))+(x-a) a_{n-1}(\omega(x)) \label{eq1}\\
&& W_{2n+1}(x) = b_n(\omega (x))+(x-a) R_{n}(\omega (x)),\;\; n \geq 0, \label{eq2}
\end{eqnarray}
where  $\:\deg
\left( a_n(x)\right) \leq n\:,\:\:\deg \left( b_n(x) \right) \leq n\:,$ and $a_{-1}(x)=0\:\:$ \cite{A:2010}.
Thus, in this general QD, we have the three free parameters $p$, $q$ and $a$.

\noindent Thus, any MPS $\{W_n\}_{n\geq 0}$ gives rise to further four sequences of polynomials by quadratic decomposition (QD), being $\{P_n\}_{n\geq 0}$ and $\{R_n\}_{n\geq 0}$ MPSs, the so-called principal components.
It is useful to assemble the four sequences produced by any QD in a matrix format as follows.
\begin{equation}\label{matrixM}
\left(
          \begin{array}{cc}
            P_{n}(x) & a_{n-1}(x) \\
            b_{n}(x) & R_{n}(x) \\
          \end{array}
        \right)
\end{equation}
A QD is said diagonal when the often called secondary component sequences $\:\left\{a_n\right\}_{n\geq 0}\:$ and $\:\left\{b_n\right\}_{n\geq 0}\:$ are null. For example, when $a=p=q=0$, the QD produced by a symmetric MPS is diagonal.

\noindent Recently, in \cite{ATZ:2018} it was presented the computational implementation of the following theoretical characterisation of the four components of a QD, with applications on specific families of orthogonal polynomials.

\begin{proposition}\cite{A:2010}
\label{prop_nova}
A MPS $\left\{W_n\right\}_{n\geq 0}\:$
defined by (\ref{divisao_ci})-(\ref{divisao}) fulfils (\ref{eq1})-(\ref{eq2}) if and only if the following recurrence relations hold.
\begin{eqnarray}
\label{pascaleq5.0} && \; b_0(x)=a-\beta_0 \\
\nonumber &&\; P_{n+1}(x)=-\displaystyle\sum_{\nu=0}^n\chi_{2n,2\nu}P_\nu(x)+\left(x-\omega(a)\right)R_n(x)
+\left(a-\beta_{2n+1}\right)b_n(x)\\
\label{pascaleq5} &&\quad \quad \quad \quad -\displaystyle\sum_{\nu=0}^{n-1}\chi_{2n,2\nu+1}b_\nu(x) \\
\nonumber &&\; a_n(x)=-\displaystyle\sum_{\nu=0}^n\chi_{2n,2\nu}a_{\nu-1}(x)-\left(a+p+\beta_{2n+1}\right)R_n(x)+b_n(x)\\
\label{pascaleq6} &&\quad \quad \quad \quad -\displaystyle\sum_{\nu=0}^{n-1}\chi_{2n,2\nu+1}R_\nu(x)\\
\nonumber && \; b_{n+1}(x)=-\displaystyle\sum_{\nu=0}^n\chi_{2n+1,2\nu+1}b_\nu(x)+\left(a-\beta_{2n+2}\right)P_{n+1}(x)\\
\label{pascaleq7}  &&\quad \quad \quad \quad +
\left(x-\omega(a)\right)a_n(x) -\displaystyle\sum_{\nu=0}^n\chi_{2n+1,2\nu}P_\nu(x)\\
\nonumber  && \; R_{n+1}(x)=-\displaystyle\sum_{\nu=0}^n\chi_{2n+1,2\nu+1}R_\nu(x)+P_{n+1}(x)-\left(a+p+\beta_{2n+2}\right)a_n(x)\\
\label{pascaleq8} &&\quad \quad \quad \quad  -\displaystyle\sum_{\nu=0}^n\chi_{2n+1,2\nu}a_{\nu-1}(x)
\end{eqnarray}
\end{proposition}

The input data of such implementation is given by the coefficients $a$, $p$ and $r$ (if we decide to make them precise), and more importantly the structure coefficients $\beta_{n}$ and $\chi_{n,\nu}$, $0\leq \nu \leq n\, ,\,$ $n \geq 0$.%

In order to investigate the structure of each component we have gathered some short commands developed in  \cite{T:2012}, namely a generic command that computes the structure coefficients of any previously defined MPS, and further instructions that search for the possibility of defining a proper MPS from the secondary components.

The symbolic implementation used has the following structure.

\begin{itemize}
\item [Step 1] Input data (adapted to a $2$-orthogonal MPS)
\medskip

\begin{tabular}{|l|l|}
\hline
\cellcolor{gray!25}SCs of $\{W_{n}(x)\}_{n \geq 0}$ &  \cellcolor{gray!25}Comments  \\
\hline
\hline
$\beta_{n}= \cdots $  &   $n \geq 0$ \\
\hline
$\chi_{n,n}  = \cdots $ & $\chi_{n,n}= \alpha_{n+1} \; , \; n \geq 0$\\
\hline
$\chi_{n,n-1} =  \cdots  $ & $\chi_{n,n-1} = \gamma_{n}\; , \; n \geq 1$  \\
\hline 
$ \chi_{n,\nu} = 0  \; , \;\; 0 \leq \nu < n-1  $ & $n \geq 1$\\
\hline
\hline
\cellcolor{gray!25}Further parameters &  \cellcolor{gray!25}Comments  \\
\hline
\hline
$nmax=\cdots$  & positive integer that indicates the higher degree \\
  & of the polynomial sets that will be computed\\
\hline
$a=\cdots; p=\cdots; q=\cdots$ & only for specific cases of the QD\\
\hline
\end{tabular}

\medskip

\item [Step 2] Application of the symbolic implementation of the QD yielding the data set
\begin{equation*}
\left(
          \begin{array}{cc}
            P_{i}(x) & a_{i-1}(x) \\
            b_{i}(x) & R_{i}(x) \\
          \end{array}
        \right) \quad i=0,\ldots, nmax
\end{equation*}
\smallskip

\item[Step 3] Application of the routine SC$_{\zeta}$   that computes the initial SCs $\beta^{\zeta}_{n}$ and $\chi^{\zeta}_{n,\nu}$ of any given MPS ${\{\zeta_{n}\}}_{n \geq 0}$, to the MPSs $\{P_{n}\}_{n \geq 0}$ and $\{R_{n}\}_{n \geq 0}$, using their initial data. 
\smallskip

\item [Step 4] Analysis of the secondary components $\{a_{n}\}_{n \geq 0}$ and $\{b_{n}\}_{n \geq 0}$. When the obtained data issues $\deg (  b_{n}(x)  )= n-k$ for a fixed $k$, $n \geq k$, and $b_{n}(x)=0\; , \; 0 \leq n < k$,  for example, it is possible to define the initial elements of the correspondent MPS $B_{n}(x)= \frac{1}{\eta_{n+k}}b_{n+k}$, $n \geq 0$, where $b_{n+k}(x)=\eta_{n+k}x^{n}+\cdots $, and apply the instruction of the third step.

\smallskip

\item [Step 5] Definition of the derivative sequence $\zeta^{[1]}_{n}(x)= \dfrac{1}{n+1} D\left( \zeta_{n+1}(x) \right) $ using recurrence \eqref{derivatives-rec} for $\zeta = P, R$ and when possible for $\zeta=A, B$, and subsequent computation of their SCs as in Step 3.

\end{itemize}

\section{Results of a case study}\label{sec:results-case-study}

The $2$-orthogonal MPS that we are going to explore in the next sections is defined by structure coefficients (SCs) that are either constants modulo two or have alternating behaviour, as next described in Table \ref{SCWn} using the constants $\beta$, $\alpha_{1}$, $\alpha_{2} \in \mathbb{C}$ and $\gamma \in \mathbb{C}\backslash \{ 0\} $, and we will consider a general QD with parameters $a$, $p$ and $q$.
\begin{table}[h]
\caption{ Structure Coefficients (SCs) of the $2$-orthogonal MPS  $\{ W_{n} (x)\}_{n \geq 0}$}
\begin{align}\label{main2MPS}
&\beta_{2n}= -(p+\beta) \; , \quad \beta_{2n+1}= \beta \; , \quad n \geq 0\, , \\
\nonumber &\chi_{2n,2n} = \alpha_{2n+1}  = \alpha_{1} \; , \quad  \chi_{2n+1,2n+1} = \alpha_{2n+2}  = \alpha_{2} \;\; , n \geq 0\, , \\
\nonumber& \chi_{n,n-1} = \gamma_{n}  = (-1)^{n} \gamma \quad , \, n \geq 1\, , \\
\nonumber & \chi_{n,\nu} = 0  \; , \quad 0 \leq \nu < n-1   \quad ,\,  n \geq 1 \, .
\end{align}
\label{SCWn}
\end{table}
Taking into account that the polynomial sets $\{a_n\}_{n\geq 0}$ and $\{b_n\}_{n\geq 0}$ may not have a MPS structure, when it is confirmed computationally that $\{a_n\}_{n= 0, \ldots, nmax}$ or $\{b_n\}_{n= 0, \ldots, nmax}$ span the vectorial space $\mathcal{P}_{max}$ of polynomial functions of degree lower or equal to $nmax$,  for a certain positive integer $nmax$, we will use the following notation.
\begin{align*}
&A_{n}(x) =\left( l_{a,n}  \right)^{-1} a_{n}(x)  \, , \quad a_{n}(x)= l_{a,n} x^n +  \, (\textrm{terms\, of\,degree\,less\,than\,}\, n )\\
&B_{n}(x) =\left( l_{b,n}  \right)^{-1} b_{n}(x)  \, , \quad b_{n}(x)= l_{b,n} x^n + \, (\textrm{terms\, of\,degree\,less\,than\,}\, n )
\end{align*}
We must begin to clarify that we have concluded through the symbolic computation of the initial SCs of the MPS $\{ W^{[1]}_{n} (x)\}_{n \geq 0}$ that $\{ W_{n} (x)\}_{n \geq 0} $ defined by \eqref{main2MPS} is not a classical $2$-orthogonal MPS.

\subsection{ Case I : $p \neq -\beta-a$}

Let us assume that the parameters $p$ and $a$ fulfill $p \neq -\beta -a$. The initial results produced by the computational implementation suggest that $\{a_{n}\}_{n \geq 0}$ is null, in particular, for all $p$ and $q$,
$$W_{2n}(x)= P_{n}\left( x^2+px+q\right)\, , \,n \geq 0\,,$$ and the remaining three polynomial sets are also $2$-orthogonal as indicated in Table \ref{Case-1}.
\begin{table}[h]
\caption{Description of the general QD of the $2$-orthogonal MPS defined by SCs \eqref{main2MPS} when $p \neq -\beta-a$.}
\begin{center}
\begin{tabular}{|l|l|}
\hline
$\{P_{n}\}_{n \geq 0}$ is $2$-orthogonal with & In general, $\{P^{[1]}_{n}\}_{n \geq 0}$ is not $d$-orthogonal  \\
recurrence coefficients \eqref{RC-P}  & for $d=1,\ldots, nmax$ \\
\hline
$\{R_{n}\}_{n \geq 0}$ is $2$-orthogonal with & $\{R^{[1]}_{n}\}_{n \geq 0}$ is $2$-orthogonal with \\
recurrence coefficients \eqref{RC-R}  & recurrence coefficients \eqref{RC-dR}  \\
\hline
$\{B_{n}\}_{n \geq 0}$ is $2$-orthogonal with & In general, $\{B^{[1]}_{n}\}_{n \geq 0}$ is not $d$-orthogonal \\
recurrence coefficients \eqref{RC-B}  & for $d=1,\ldots, nmax$ \\
\hline
\end{tabular}
\end{center}
\label{Case-1}
\end{table}%
\begin{align}\label{RC-P}
&\beta_{0}^{P}= q+\alpha_{1} + \left(p+\beta \right)\beta\; , \quad \beta_{n}^{P}= q+\alpha_{1} +\alpha_{2} + \left(p+\beta \right)\beta \; , \; n \geq 1\\
\nonumber  &\alpha_{n}^{P} = \alpha_{1} \alpha_{2} + \gamma(p+2\beta) \; , \quad \gamma_{n}^{P} = \gamma^{2} \; , \; n \geq 1. \\
\nonumber &\\
\label{RC-R}
& \beta_{n}^{R}= q+\alpha_{1} +\alpha_{2} + \left(p+\beta \right)\beta \; , \; n \geq 0\\
\nonumber  &\alpha_{n}^{R} = \alpha_{1} \alpha_{2} + \gamma(p+2\beta) \; , \quad \gamma_{n}^{R} = \gamma^{2} \; , \; n \geq 1. \\
\nonumber &\\
\label{RC-B} 
& \beta_{0}^{B}= q+ \alpha_{1} +\alpha_{2}+ \frac{a \beta \left(p+ \beta \right) +\beta(p+\beta)^{2}  -\gamma }{a+p+\beta}\\
\nonumber  &\beta_{n}^{B}= q+\alpha_{1} +\alpha_{2} + \left(p+\beta \right)\beta \; , \,n \geq 1\\
\nonumber  &\alpha_{n}^{B} = \alpha_{1} \alpha_{2} + \gamma(p+2\beta) \; , \quad \gamma_{n}^{B} =  \gamma^{2} \; , \; n \geq 1. 
\end{align}
\begin{align} \label{RC-dR} 
& \beta_{n}^{R^{[1]}}= q+\alpha_{1} +\alpha_{2} + \left(p+\beta \right)\beta \; , n \geq 0\\
\nonumber  &\alpha_{n}^{R^{[1]}} = \frac{n(n+3)}{(n+1)(n+2)}\left(  \alpha_{1} \alpha_{2} + \gamma(p+2\beta)\right) \; , \quad\\
\nonumber & \gamma_{n}^{R^{[1]}} =  \frac{n(n+5)}{(n+2)(n+3)}\gamma^{2} \; , \; n \geq 1. 
\end{align}

\subsubsection{ Particular case : $p \neq -\beta-a$ and $\alpha_{2}=0$}\label{subsubsec:Case2}

Let us continue to assume that the parameters $p$ and $a$ fulfill $p \neq -\beta -a$. Some of the SCs obtained previously indicated the presence of the factor $\alpha_{2}$, thus it seemed relevant to collect the same data when $\alpha_{2}=0$.
As a particular case of the previous one, most of the outlined features are maintained, as for instance, $W_{2n}(x)= P_{n}\left( x^2+px+q\right)\, , \,n \geq 0\,$. The summary can be checked in Table \ref{Case-2}.

\begin{table}[h]
\caption{Description of the general QD of the $2$-orthogonal MPS defined by SCs \eqref{main2MPS} when $p \neq -\beta-a$ and $\alpha_{2}=0$.}
\begin{center}
\begin{tabular}{|l|l|}
\hline
$\{P_{n}\}_{n \geq 0}$ is $2$-orthogonal and  & $\{P^{[1]}_{n}\}_{n \geq 0}$ is $2$-orthogonal and \\
identical to $\{R_{n}\}_{n \geq 0}$  & identical to  $\{R^{[1]}_{n}\}_{n \geq 0}$ \\
\hline
$\{R_{n}\}_{n \geq 0}$ is $2$-orthogonal with & $\{R^{[1]}_{n}\}_{n \geq 0}$ is $2$-orthogonal with \\
recurrence coefficients \eqref{RC-R-alpha2=0}  & recurrence coefficients \eqref{RC-dR-alpha2=0}  \\
\hline
$\{B_{n}\}_{n \geq 0}$ is $2$-orthogonal with & In general, $\{B^{[1]}_{n}\}_{n \geq 0}$ is not $d$-orthogonal \\
recurrence coefficients \eqref{RC-B-alpha2=0}  & for $d=1,\ldots,nmax$ \\
\hline
\end{tabular}
\end{center}
\label{Case-2}
\end{table}%
\begin{align}
\label{RC-R-alpha2=0}
& \beta_{n}^{R}= q+\alpha_{1}  + \left(p+\beta \right)\beta \; , n \geq 0\\
\nonumber  &\alpha_{n}^{R} = \gamma(p+2\beta) \; , \quad \gamma_{n}^{R} = \gamma^{2} \; , \; n \geq 1. \\
\nonumber &\\
\label{RC-B-alpha2=0} 
& \beta_{0}^{B}=q+ \alpha_{1}+\frac{  a\beta \left(p+\beta \right) +\beta(p+\beta)^{2}  -\gamma}{a+p+\beta}\\
\nonumber  &\beta_{n}^{B}= q+\alpha_{1} + \left(p+\beta \right)\beta \; , \; n \geq 1\\
\nonumber  &\alpha_{n}^{B} =  \gamma(p+2\beta) \; , \quad \gamma_{n}^{B} =  \gamma^{2} \; , \; n \geq 1. 
\end{align}
\begin{align} \label{RC-dR-alpha2=0} 
& \beta_{n}^{R^{[1]}}= q+\alpha_{1} + \left(p+\beta \right)\beta \; , n \geq 0\\
\nonumber  &\alpha_{n}^{R^{[1]}} = \frac{n(n+3)}{(n+1)(n+2)}  \gamma(p+2\beta) \; , \\
&\nonumber  \gamma_{n}^{R^{[1]}} =  \frac{n(n+5)}{(n+2)(n+3)}\gamma^{2} \; , \; n \geq 1. 
\end{align}

\subsection{ Case II : $p = -\beta-a$}

Let us now assume that the parameters $p$ and $a$ fulfill $p = -\beta -a$. The initial results produced by the computational implementation continue to suggest that $\{a_{n}\}_{n \geq 0}$ is null, in particular, $W_{2n}(x)= P_{n}\left( x^2+px+q\right)\, , \,n \geq 0\,,$ and the remaining information is indicated in Table \ref{Case-3}.
In this case, the  polynomial set $\{b_{n}(x)\}_{n \geq 0}$ fulfils $\deg\left( b_{n}(x)  \right)=n-1\, , \, n \geq 1$ and $b_{0}(x)=0$, with $b_{n}(x)= \gamma x^{n-1} +  \, (\textrm{terms\, of\,degree\,less\,than\,}\, n-1 )$,  and therefore we define the MPS
$$\bar{B}_{n}(x)= \gamma^{-1} b_{n+1}(x)\, , \, n \geq 0.$$
\begin{table}[h]
\caption{Description of the general QD of the $2$-orthogonal MPS defined by SCs \eqref{main2MPS} when $p = -\beta-a$.}
\begin{center}
\begin{tabular}{|l|l|}
\hline
$\{P_{n}\}_{n \geq 0}$ is $2$-orthogonal with & In general, $\{P^{[1]}_{n}\}_{n \geq 0}$ is not $d$-orthogonal\\
recurrence coefficients \eqref{RC-P-II}  & for $d=1,\ldots, nmax$ \\
\hline
$\{R_{n}\}_{n \geq 0}$ is $2$-orthogonal with & $\{R^{[1]}_{n}\}_{n \geq 0}$ is $2$-orthogonal with \\
recurrence coefficients \eqref{RC-R-II}  & recurrence coefficients \eqref{RC-dR-II}  \\
\hline
$\{\bar{B}_{n}\}_{n \geq 0}$ is $2$-orthogonal and & $\{\bar{B}^{[1]}_{n}\}_{n \geq 0}$ is $2$-orthogonal and \\
identical to $\{R_{n}\}_{n \geq 0}$   & identical to  $\{R^{[1]}_{n}\}_{n \geq 0}$\\
\hline
\end{tabular}
\end{center}
\label{Case-3}
\end{table}%

\begin{align}\label{RC-P-II}
&\beta_{0}^{P}= q+\alpha_{1} -a\beta\; , \quad \beta_{n}^{P}= q+\alpha_{1} +\alpha_{2} -a \beta \; , n \geq 1\\
\nonumber  &\alpha_{n}^{P} = \alpha_{1} \alpha_{2} - a \gamma + \beta \gamma \; , \quad \gamma_{n}^{P} = \gamma^{2} \; , \; n \geq 1. \\
\nonumber &\\
\label{RC-R-II}
& \beta_{n}^{R}= q+\alpha_{1} +\alpha_{2} -a\beta \; , n \geq 0\\
\nonumber  &\alpha_{n}^{R} = \alpha_{1} \alpha_{2} - a \gamma +\beta\gamma \; , \quad \gamma_{n}^{R} = \gamma^{2} \; , \; n \geq 1. 
\end{align}
\begin{align} \label{RC-dR-II} 
& \beta_{n}^{R^{[1]}}= q+\alpha_{1} +\alpha_{2} -a\beta \; , n \geq 0\\
\nonumber  &\alpha_{n}^{R^{[1]}} = \frac{n(n+3)}{(n+1)(n+2)}\left(  \alpha_{1} \alpha_{2} - a\gamma+\beta\gamma \right) \; , \\
&\nonumber  \gamma_{n}^{R^{[1]}} =  \frac{n(n+5)}{(n+2)(n+3)}\gamma^{2} \; , \; n \geq 1. 
\end{align}
More precisely, when $p = -\beta-a$ we get the following general QD:
\begin{align*}
&W_{2n}(x)  = P_{n}(x^2+px+q)\\
&W_{2n+1}(x)  = \gamma R_{n-1}(x^2+px+q)+(x-a)R_{n}(x^2+px+q)
\end{align*}
where $\{R_{n}(x)\}_{n \geq 0}$ is a 2-classical MPS and $\{P_{n}(x)\}_{n \geq 0}$ is a co-recursive of $\{R_{n}(x)\}_{n \geq 0}$.

\subsubsection{ Particular case : $p = -\beta-a$ and $\alpha_{2}=0$}

Comparing the coefficients described in \eqref{RC-P-II}-\eqref{RC-R-II}, we easily confirm that if $\alpha_{2}=0$, the principal components $\{P_{n}\}_{n \geq 0}$ and $\{R_{n}\}_{n \geq 0}$ coincide and similarly to subsection \ref{subsubsec:Case2} , we conclude that $\{P^{[1]}_{n}\}_{n \geq 0}$ is also $2$-orthogonal. Moreover, $\{P_{n}\}_{n \geq 0}$, $\{R_{n}\}_{n \geq 0}$ and $\{\bar{B}_{n}\}_{n \geq 0}$ are identical. The correspondent SCs can be obtained through \eqref{RC-R-II} for $\{R_{n}\}_{n \geq 0}$ and  \eqref{RC-dR-II} for $\{R^{[1]}_{n}\}_{n \geq 0}$ replacing $\alpha_{2}$ by zero.

\section{Results for a co-recursive sequence and for higher perturbations}

Let us now consider a co-recursive sequence of the $2$-orthogonal MPS  with SCs \eqref{main2MPS} $\{ W_{n} (x)\}_{n \geq 0}$ as defined previously in \eqref{RR2ortho-tilde} by a single change in the parameter $\beta_{0}$. More precisely, we will perform the quadratic decomposition of a MPS defined by the SCs \eqref{main2MPS-co}  and follow the same notation and definitions of the Section \ref{sec:results-case-study}. Further, we assume that $\tau \neq -p-\beta$ otherwise we would be facing the MPS  $\{ W_{n} (x)\}_{n \geq 0}$ once more. 
\begin{align}\label{main2MPS-co}
&\beta_{0}= \tau \; , \quad \beta_{2n}= -(p+\beta) \; , \; n \geq 1\,, \\
\nonumber &\beta_{2n+1}= \beta \; , \quad n \geq 0\, , \\
\nonumber &\chi_{2n,2n} = \alpha_{2n+1}  = \alpha_{1} \; , \quad  \chi_{2n+1,2n+1} = \alpha_{2n+2}  = \alpha_{2} \;\; , n \geq 0\, , \\
\nonumber& \chi_{n,n-1} = \gamma_{n}  = (-1)^{n} \gamma \quad , \, n \geq 1\, , \\
\nonumber & \chi_{n,\nu} = 0  \; , \quad 0 \leq \nu < n-1   \quad ,\,  n \geq 1 \, .
\end{align}

\subsection{ Case I : $\tau \neq a$ and $\tau \neq -p-\beta$}

Let us assume that the parameters $\tau$ and $a$ fulfill $\tau \neq a$ besides the obvious restriction $\tau \neq -p-\beta$. The initial results produced by the computational implementation no longer suggest that $\{a_{n}\}_{n \geq 0}$ is null and therefore we have now the four components to analyse. The essential information is issued in Table \ref{Case-1-co}.
\bigskip

\begin{table}[h]
\caption{Description of the general QD of the $2$-orthogonal MPS defined by SCs \eqref{main2MPS-co} when $\tau \neq a$ and $\tau \neq -p-\beta$.}
\begin{center}
\begin{tabular}{|l|l|}
\hline
$\{P_{n}\}_{n \geq 0}$ is $2$-orthogonal with & In general, $\{P^{[1]}_{n}\}_{n \geq 0}$ is not $d$-orthogonal \\
recurrence coefficients \eqref{RC-P-co}  & for $d=1,\dots,nmax$\\
\hline
$\{R_{n}\}_{n \geq 0}$ is $2$-orthogonal with & $\{R^{[1]}_{n}\}_{n \geq 0}$ is $2$-orthogonal with \\
recurrence coefficients \eqref{RC-R}  & recurrence coefficients \eqref{RC-dR}  \\
\hline
$\{A_{n}\}_{n \geq 0}$ is $2$-orthogonal and  & $\{A^{[1]}_{n}\}_{n \geq 0}$ is  $2$-orthogonal  and \\
coincides with  $\{R_{n}\}_{n \geq 0}$  & coincides with  $\{R^{[1]}_{n}\}_{n \geq 0}$  \\
\hline
$\{B_{n}\}_{n \geq 0}$ is $2$-orthogonal with & In general, $\{B^{[1]}_{n}\}_{n \geq 0}$ is not $d$-orthogonal \\
recurrence coefficients \eqref{RC-B-co}  & for $d=1,\dots,nmax$  \\
\hline
\end{tabular}
\end{center}
\label{Case-1-co}
\end{table}%

\begin{align}
\nonumber &\beta_{0}^{P}= ap+q+\alpha_{1} +a\beta +a\tau - \beta\tau \; , \quad \beta_{n}^{P}= q+\alpha_{1} +\alpha_{2} + \left(p+\beta \right)\beta \; , n \geq 1\\
\label{RC-P-co} &\alpha_{1}^{P} = \alpha_{1} \alpha_{2} + \gamma(\beta-\tau) \; , \quad \alpha_{n}^{P}=\alpha_{1} \alpha_{2} + \gamma(p+2\beta)  \\
\nonumber  &\gamma_{n}^{P} = \gamma^{2} \; , \; n \geq 1. \\
\nonumber &\\
\nonumber & \beta_{0}^{B}=q+\alpha_{2}+p \beta+ \beta^2 + \frac{\alpha_{1}\left(a+p+\beta  \right)-\gamma }{a-\tau}\\
\label{RC-B-co}   &\beta_{n}^{B}= q+\alpha_{1} +\alpha_{2} + \left(p+\beta \right)\beta \; , n \geq 1\\
\nonumber  &\alpha_{n}^{B} = \alpha_{1} \alpha_{2} + \gamma(p+2\beta) \; ,\\
\nonumber & \gamma_{n}^{B} =  \gamma^{2} \; , \; n \geq 1. 
\end{align}

\subsection{ Case II : $\tau = a$ and $\tau \neq -p-\beta$}

Let us finally assume that the parameters $\tau$ and $a$ are identical. The secondary component  $\{a_{n}\}_{n \geq 0}$ is not zero and therefore we have now the four components to analyse. Most of the information obtained in the previous case I is kept. The difference lies in the sequence $\{b_{n}\}_{n \geq 0}$ for which we get   $\deg\left( b_{n}(x)  \right)=n-1\, , \, n \geq 1$ and $b_{0}(x)=0$, with 
\newline $b_{n}(x) =\left( \gamma-\alpha_{1}(a+p+\beta) \right) x^{n-1} +  \, (\textrm{terms\, of\,degree\,less\,than\,}\, n-1 )
$ and therefore we define the MPS
$$\bar{B}_{n}(x)= \left( \gamma-\alpha_{1}(a+p+\beta) \right)^{-1} b_{n+1}(x)\, , \, n \geq 0.$$
The essential information is indicated in Table \ref{Case-2-co}.
\begin{table}[h]
\caption{Description of the general QD of the $2$-orthogonal MPS defined by SCs \eqref{main2MPS-co} when $\tau = a$ and $\tau \neq -p-\beta$.}
\begin{center}
\begin{tabular}{|l|l|}
\hline
$\{P_{n}\}_{n \geq 0}$ is $2$-orthogonal with & In general, $\{P^{[1]}_{n}\}_{n \geq 0}$ is not $d$-orthogonal \\
recurrence coefficients \eqref{RC-P-co-tauzero}  & for $d=1,\ldots, nmax$ \\
\hline
$\{R_{n}\}_{n \geq 0}$ is $2$-orthogonal with & $\{R^{[1]}_{n}\}_{n \geq 0}$ is $2$-orthogonal with \\
recurrence coefficients \eqref{RC-R}  & recurrence coefficients \eqref{RC-dR}  \\
\hline
$\{A_{n}\}_{n \geq 0}$ is $2$-orthogonal and  & $\{A^{[1]}_{n}\}_{n \geq 0}$ is  $2$-orthogonal  and \\
coincides with  $\{R_{n}\}_{n \geq 0}$  & coincides with  $\{R^{[1]}_{n}\}_{n \geq 0}$  \\
\hline
$\{\bar{B}_{n}\}_{n \geq 0}$ is $2$-orthogonal and & $\{\bar{B}^{[1]}_{n}\}_{n \geq 0}$ is $2$-orthogonal  and \\
coincides with  $\{R_{n}\}_{n \geq 0}$ & coincides with  $\{R^{[1]}_{n}\}_{n \geq 0}$\\
\hline
\end{tabular}
\end{center}
\label{Case-2-co}
\end{table}%

\begin{align}\label{RC-P-co-tauzero}
&\beta_{0}^{P}= a^2+ ap+q+\alpha_{1}  \; , \quad \beta_{n}^{P}= q+\alpha_{1} +\alpha_{2} + \left(p+\beta \right)\beta \; , n \geq 1\\
\nonumber  &\alpha_{1}^{P} = \alpha_{1} \alpha_{2} + \gamma(\beta-a) \; , \quad \alpha_{n}^{P}=\alpha_{1} \alpha_{2} + \gamma(p+2\beta)  \\
\nonumber  &\gamma_{n}^{P} = \gamma^{2} \; , \; n \geq 1. \\
\nonumber &
\end{align}
\subsection{Perturbation of order  two (I)}

The main structure of the results described in the previous sections holds  when we consider a larger perturbation of order two, given the non-zero constants $\eta_{1}\, , \eta_{2}$ and $\xi$, as follows.
\begin{align}\label{main2MPS-pert-order2-I}
&\beta_{0}= \tau \; , \quad \beta_{2n}= -(p+\beta) \; , \; n \geq 1\,, \\
\nonumber &\beta_{2n+1}= \beta \; , \quad n \geq 0\, , \\
\nonumber & \chi_{0,0}=\alpha_{1}\eta_{1}\, ,\; \chi_{1,1}=\alpha_{2}\eta_{2}\, ,\; \\
\nonumber &\chi_{2n,2n} = \alpha_{2n+1}  = \alpha_{1} \; , \quad  \chi_{2n+1,2n+1} = \alpha_{2n+2}  = \alpha_{2} \;\; , n \geq 1\, , \\
\nonumber& \chi_{1,0} =  - \gamma \xi \quad , \, \quad \chi_{n,n-1} = \gamma_{n}  = (-1)^{n} \gamma \quad , \, n \geq 2\, , \\
\nonumber & \chi_{n,\nu} = 0  \; , \quad 0 \leq \nu < n-1   \quad ,\,  n \geq 1 \, .
\end{align}
More precisely, the principal component sequences $\{ P_{n} \}_{n \geq 0}$ and $\{ R_{n} \}_{n \geq 0}$ are $2$-orthogonal and

\begin{align}
\nonumber &\beta_{0}^{P}= a p + q + a \beta + \alpha_{1}\eta_{1} + \left( a-\beta\right) \tau   \; , \quad \beta_{1}^{P}=  q + \alpha_{1}+\left( p+\beta \right) \beta + \alpha_{2}\eta_{2} \, ,\\
\nonumber &\beta_{n}^{P}= q+\alpha_{1} +\alpha_{2} + \left(p+\beta \right)\beta \; ,\; n \geq 2\, ,\\
\nonumber  &\alpha_{1}^{P} = a\gamma +\alpha_{1} \alpha_{2} \eta_{1}\eta_{2}+\gamma \xi \left( \beta-a \right)-\gamma\tau\; , \quad \alpha_{n}^{P}=\alpha_{1} \alpha_{2} + \gamma(p+2\beta)  \, ,\\
\nonumber  & \gamma_{1}^{P} =\gamma\left(a\alpha_{2}- a\alpha_{2}\eta_{2}+\gamma\xi-\alpha_{2}\tau+\alpha_{2}\eta_{2}\tau \right)\; , \quad
\gamma_{n}^{P} = \gamma^{2} \; , \; n \geq 2;
\end{align}
\begin{align}
\nonumber &\beta_{0}^{R}=  q + \left( p+\beta  \right)\beta + \alpha_{1}\eta_{1} + \alpha_{2}\eta_{2}   \; , \quad \beta_{n}^{R}= q+\alpha_{1} +\alpha_{2} + \left(p+\beta \right) \beta \; , n \geq 1\, , \\
\nonumber  &\alpha_{1}^{R} = \gamma\left( p+2\beta \right) +\alpha_{1} \alpha_{2}\eta_{2}\; , \quad \alpha_{n}^{R}=\alpha_{1} \alpha_{2} + \gamma(p+2\beta) \; , n \geq 2\, , \\
\nonumber  &
\gamma_{n}^{R} = \gamma^{2} \; , \; n \geq 1. 
\end{align}

With regard to the sequences  $\{a_{n}\}_{n \geq 0}$ and $\{b_{n}\}_{n \geq 0}$  we get   $\deg\left( a_{n}(x)  \right)=n\, , \, n \geq 0$ with $l_{a,n}=-p-\beta-\tau$ and $\deg\left( b_{n}(x)  \right)=n\, , \, n \geq 0$ with $l_{a,n}=a-\tau$. Supposing $\tau \neq -p-\beta$ and $\tau \neq a$, we define the MPSs
$A_{n}(x)= l_{a,n}^{-1} a_{n}(x)$ and $B_{n}(x)= l_{b,n}^{-1} b_{n}(x)\, , \, n \geq 0$ and we obtain
\begin{align}
\nonumber &\beta_{0}^{A}= q+ \alpha_{1} +\alpha_{2}\eta_{2}+ \frac{p\beta\left( p+2\beta\right)+\beta^3-\gamma+\gamma\xi+\tau\left( p\beta+\beta^2 \right)}{p+\beta+\tau}\\
\nonumber &\beta_{n}^{A}= q+\alpha_{1} +\alpha_{2} + \left(p+\beta \right) \beta \; , n \geq 1\, , \\
\nonumber  &\alpha_{1}^{A} =\alpha_{1} \alpha_{2}+  \frac{\gamma \left(p^2- \alpha_{2} +3p\beta+2\beta^2+\alpha_{2}\eta_{2}+\tau\left(p+2\beta \right)
\right) }{p+\beta+\tau}\\
\nonumber & \alpha_{n}^{A}= \alpha_{1} \alpha_{2} + \gamma(p+2\beta)  \; , n \geq 2\, , \\
\nonumber  & \gamma_{n}^{A} = \gamma^{2} \; , \; n \geq 1\, ;\\
\nonumber &\\
\nonumber &\beta_{0}^{B}= q+p\beta +\beta^2+\alpha_{2}\eta_{2}+ \frac{\alpha_{1}\eta_{1}\left(a+ p+\beta\right)-\gamma\xi}{a-\tau}\\
\nonumber &\beta_{n}^{B}= q+\alpha_{1} +\alpha_{2} + \left(p+\beta \right) \beta \; , n \geq 1\, , \\
\nonumber  &\alpha_{1}^{B} =  \gamma(p+2\beta) +\alpha_{1} \alpha_{2}\eta_{2}+  \frac{\gamma\alpha_{1} \left(\eta_{1}-\xi \right) }{a-\tau}\, , \\
\nonumber & \alpha_{n}^{B}= \alpha_{1} \alpha_{2} + \gamma(p+2\beta)  \; , n \geq 2\, , \\
\nonumber  & \gamma_{n}^{B} = \gamma^{2} \; , \; n \geq 1. 
\end{align}
Finally, we remark that the data obtained allow us to assert that none of the four $2$-orthogonal MPSs are in general $2$-classical.
The alternative hypothesis $\tau =a$ was also treated with analogous conclusions, except for the sequence $\{b_{n}\}_{n \geq 0}$  for which we get   $\deg\left( b_{n}(x)  \right)=n-1\, , \, n \geq 0$ with 
\newline $b_{n}(x) =\left(\gamma \xi -\alpha_{1}\left( a+p+\beta\right) \eta_{1} \right) x^{n-1} +  \, (\textrm{terms\, of\,degree\,less\,than\,}\, n-1 )$ and thus further suppositions must be taken in order to define a MPS $\{\bar{B}_{n}\}_{n \geq 0}$ as we did before.

\subsection{Perturbation of order  two (II)}
Let us now consider another perturbation of order two, where only the first two $\beta$s are altered by means of the constants $\tau_{1} $ and $\tau_{2}$ assuming that $\tau_{1}\neq -p-\beta$ and $\tau_{2} \neq \beta$.
\begin{align}\label{main2MPS-pert-order2-II}
&\beta_{0}= \tau_{1} \; , \; \beta_{1}= \tau_{2}\, , \quad \beta_{2n}= -(p+\beta) \; , \; \beta_{2n+1}= \beta\;, \; n \geq 1\,, \\
\nonumber &\chi_{2n,2n} = \alpha_{2n+1}  = \alpha_{1} \; , \quad  \chi_{2n+1,2n+1} = \alpha_{2n+2}  = \alpha_{2} \;\; , n \geq 0\, , \\
\nonumber& \chi_{n,n-1} = \gamma_{n}  = (-1)^{n} \gamma \quad , \, n \geq 1\, , \\
\nonumber & \chi_{n,\nu} = 0  \; , \quad 0 \leq \nu < n-1   \quad ,\,  n \geq 1 \, .
\end{align}
We get for the sequences $\{ P_{n} \}_{n \geq 0}$ and $\{ R_{n} \}_{n \geq 0}$:
\begin{align}
\nonumber &\beta_{0}^{P}= a p + q + \alpha_{1}+ a \left( \tau_{1}+\tau_{2}\right)-\tau_{1}\tau_{2}  \, ,\\
\nonumber &\beta_{n}^{P}= q+\alpha_{1} +\alpha_{2} + \left(p+\beta \right)\beta \; ,\; n \geq 1\, ,\\
\nonumber &\alpha_{1}^{P} = \alpha_{1}\alpha_{2} -a\alpha_{2}\beta  +\beta\gamma + \alpha_{2}\beta \tau_{1}-\gamma \tau_{1} + \alpha_{2}\tau_{2}\left(a-\tau_{1}\right)\; , \quad \\
\nonumber & \alpha_{n}^{P}=\alpha_{1} \alpha_{2} + \gamma(p+2\beta)  \; , \; n \geq 2 \, ,\quad \gamma_{n}^{P} = \gamma^{2} \; , \; n \geq 1;\\
\nonumber &\\
\nonumber &\beta_{0}^{R}=  q + \alpha_{1}+\alpha_{2}+\beta\left( p+\tau_{1}+\tau_{2}  \right) -\tau_{1}\tau_{2}    \; , \\
\nonumber &\beta_{n}^{R}= q+\alpha_{1} +\alpha_{2} + \left(p+\beta \right) \beta \, ,\,  n \geq 1\, , \\
\nonumber  & \alpha_{1}^{R}=\alpha_{1} \alpha_{2} + \gamma(p+\beta+\tau_{2}) \; , \\
\nonumber  & \alpha_{n}^{R}=\alpha_{1} \alpha_{2} + \gamma(p+2\beta) \; , \, n \geq 2\, , \quad \gamma_{n}^{R} = \gamma^{2} \; , \; n \geq 1. 
\end{align}
With regard to the sequences  $\{a_{n}\}_{n \geq 0}$ and $\{b_{n}\}_{n \geq 0}$  we get   $\deg\left( a_{n}(x)  \right)=n\, , \, n \geq 0\, ,$ with $l_{a,n}=-p-\tau_{1}-\tau_{2}$ and $\deg\left( b_{n}(x)  \right)=n\, , \, n \geq 0\, ,$ with $l_{b,0}=a-\tau_{1}\, $ and $l_{b,n}=a+\beta-\tau_{1}-\tau_{2}$. Supposing also that $\tau_{1} \neq a,$ and $\tau_{1}+\tau_{2} \neq a+\beta, -p$, we define the MPSs
$A_{n}(x)= l_{a,n}^{-1} a_{n}(x)\, , \, n \geq 0,$ and $B_{n}(x)= l_{b,n}^{-1} b_{n}(x)\, , \, n \geq 0$ and we obtain
\begin{align}
\nonumber &\beta_{0}^{A}= q+ \alpha_{1} + \frac{\alpha_{2}\left( p+\tau_{1} +\beta \right)+p^2\beta+p\beta^2 +\left( \tau_{1}+\tau_{2} \right)\left( p\beta+\beta^2 \right)}{p+\tau_{1}+\tau_{2}}\, ,\\
\nonumber &\beta_{n}^{A}= q+\alpha_{1} +\alpha_{2} + \left(p+\beta \right) \beta \; , n \geq 1\, , \\
\nonumber & \alpha_{n}^{A}= \alpha_{1} \alpha_{2} + \gamma(p+2\beta)  \; , n \geq 1\, , \\
\nonumber  & \gamma_{n}^{A} = \gamma^{2} \; , \; n \geq 1\, ;\\
\nonumber &\\
\nonumber &\beta_{0}^{B}=q+ \frac{a\left(\alpha_{1}+\alpha_{2} \right)+ p\left( \alpha_{1} +a\beta \right)+\beta\alpha_{1}-\gamma+ a \beta \left( \tau_{1}+\tau_{2}\right)-\tau_{1}\alpha_{2}- \tau_{1}\tau_{2}\left( a+p+\beta \right) }{a+\beta-\tau_{1}-\tau_{2}}\\
\nonumber &\beta_{n}^{B}= q+\alpha_{1} +\alpha_{2} + \left(p+\beta \right) \beta \; , n \geq 1\, , \\
\nonumber  & \alpha_{1}^{B} = \frac{\left(a- \tau_{1}\right) \left(\alpha_{1}\alpha_{2}+\gamma\left( p+\beta+\tau_{2} \right) \right) }{a+\beta-\tau_{1}-\tau_{2}}\, , \quad  \alpha_{n}^{B}= \alpha_{1} \alpha_{2} + \gamma(p+2\beta)  \; , n \geq 2\, , \\
\nonumber  & \gamma_{1}^{B} = \frac{\gamma^2\left(a- \tau_{1}\right) }{a+\beta-\tau_{1}-\tau_{2}}\, , \quad \gamma_{n}^{B} = \gamma^{2} \; , \; n \geq 2. 
\end{align}
We finally conclude that none of the four $2$-orthogonal MPSs is $2$-classical.

\section{Analytical overview}

Let us now rewrite Proposition \eqref{prop_nova} taking into account the particular SCs of a $2$-orthogonal MPS  $\chi_{n,n}= \alpha_{n+1} \; , \; n \geq 0$, 
$\chi_{n,n-1} = \gamma_{n}\; , \; n \geq 1$ and 
$ \chi_{n,\nu} = 0  \; , \;\; 0 \leq \nu < n-1  \, , $ $n \geq 0$.

\begin{eqnarray}
\label{2orto-qd-0} && \; b_0(x)=a-\beta_0\, ,\, \\ [5pt]
\label{2orto-qd-2.6} &&\; P_{n+1}(x)= -\alpha_{2n+1}P_n(x)+\left(x-\omega(a)\right)R_n(x)
\\
\nonumber &&\hspace{4cm}+\left(a-\beta_{2n+1}\right)b_n(x)- \gamma_{2n}b_{n-1}(x)\, ,\,  \\ [5pt]
\label{2orto-qd-2.7} &&\; a_n(x)=-\alpha_{2n+1}a_{n-1}(x)-\left(a+p+\beta_{2n+1}\right)R_n(x)\\
\nonumber &&\hspace{4cm}+b_n(x)-\gamma_{2n}R_{n-1}(x)\, ,\,\\ [5pt]
\label{2orto-qd-2.8} && \; b_{n+1}(x)=-\alpha_{2n+2}b_n(x)+\left(a-\beta_{2n+2}\right)P_{n+1}(x) \\
\nonumber  &&\hspace{4cm}+
\left(x-\omega(a)\right)a_n(x) -\gamma_{2n+1}P_n(x)\, ,\,\\  [5pt]
\label{2orto-qd-2.9} && \; R_{n+1}(x)=-\alpha_{2n+2}R_n(x)+P_{n+1}(x)-\left(a+p+\beta_{2n+2}\right)a_n(x) \\
\nonumber  &&\hspace{4cm} -\gamma_{2n+1}a_{n-1}(x)\, ,\,\\
\nonumber  &&\textrm{with}\; n \geq 0\,  \; \textrm{and}\;R_{-1}(x)=a_{-1}(x) = b_{-1}(x)=0.
\end{eqnarray}
\begin{lemma}
The general quadratic decomposition of any $2$-orthogonal MPS is not diagonal.
\end{lemma}
\begin{proof}
If we suppose a diagonal QD, identity \eqref{2orto-qd-2.7} reads
$$0=-\left(a+p+\beta_{2n+1}\right)R_n(x)-\gamma_{2n}R_{n-1}(x)$$
implying $a+p+\beta_{2n+1}=0$ and $\gamma_{2n}=0$. Since this latest contradicts the regularity, the proof is complete. 
\end{proof}

Remarking that identities \eqref{2orto-qd-2.7} and \eqref{2orto-qd-2.9} allow us to express $b_{n}(x)$ and $P_{n+1}(x)$ in terms of elements of $\{ R_{n} \}_{n\geq 0}$ and $\{ a_{n} \}_{n\geq 0}$:
\begin{align*}
&b_n(x)= a_n(x)+\alpha_{2n+1}a_{n-1}(x)+\left(a+p+\beta_{2n+1}\right)R_n(x)+\gamma_{2n}R_{n-1}(x)\\
&P_{n+1}(x)=R_{n+1}(x)+\alpha_{2n+2}R_n(x)+\left(a+p+\beta_{2n+2}\right)a_n(x)+\gamma_{2n+1}a_{n-1}(x)
\end{align*}
we are able to transform \eqref{2orto-qd-2.8} into the following relation:
\begin{align}
\nonumber &a_{n+1}(x)-\Big(x-\omega(a)+\left(a-\beta_{2n+2}\right)\left(a+p+\beta_{2n+2}\right) - \alpha_{2n+3}-\alpha_{2n+2}\Big) a_n(x)\\
\nonumber &+\Big(\alpha_{2n+2}\alpha_{2n+1} + \gamma_{2n+1}\left(p+\beta_{2n+2}+\beta_{2n}\right) \Big)a_{n-1}(x) +\gamma_{2n+1}\gamma_{2n-1}a_{n-2}(x)\\
\label{as-rs} &+ \left( p+\beta_{2n+3}+\beta_{2n+2} \right) R_{n+1}(x)\\
\nonumber&+ \Big(\gamma_{2n+2}+\gamma_{2n+1}+\alpha_{2n+2} \left( p+\beta_{2n+2}+\beta_{2n+1} \right)  \Big)R_{n}(x)\\
\nonumber& +\Big(\alpha_{2n+2} \gamma_{2n}+\gamma_{2n+1}\alpha_{2n} \Big)R_{n-1}(x)=0\, ,\; n \geq 1.
\end{align}
Likewise, identity \eqref{2orto-qd-2.6} implies the next relation:
\begin{align}
\nonumber&R_{n+1}(x)-\Big(x-\omega(a) + \left(a-\beta_{2n+1}\right)\left(a+p+\beta_{2n+1}\right)  -\alpha_{2n+2}-\alpha_{2n+1}\Big) R_n(x)\\
\nonumber &+\Big(\alpha_{2n+1}\alpha_{2n} + \gamma_{2n}\left(p+\beta_{2n+1}+\beta_{2n-1}\right)\Big)R_{n-1}(x)+\gamma_{2n}\gamma_{2n-2}R_{n-2}(x) \\
\label{rs-as}&+ \left( p+\beta_{2n+2}+\beta_{2n+1} \right) a_{n}(x)\\
\nonumber &+ \Big(\gamma_{2n+1}+\gamma_{2n}+\alpha_{2n+1} \left( p+\beta_{2n+1}+\beta_{2n} \right)  \Big)a_{n-1}(x)\\
\nonumber& +\Big(\alpha_{2n+1} \gamma_{2n-1}+\gamma_{2n}\alpha_{2n-1} \Big)a_{n-2}(x)=0\, ,\; n \geq 1.
\end{align}
On the other hand, identities \eqref{2orto-qd-2.6} and \eqref{2orto-qd-2.8} allow us to express $\left(x-\omega(a)\right)R_{n}(x)$ and $\left(x-\omega(a)\right)a_{n}(x)$ in terms of elements of $\{ P_{n} \}_{n\geq 0}$ and $\{ b_{n} \}_{n\geq 0}$:
\begin{align*}
& \left(x-\omega(a)\right) R_n(x)= -\left(a-\beta_{2n+1} \right) b_n(x)+\gamma_{2n}b_{n-1}(x)+P_{n+1}(x)+\alpha_{2n+1}P_{n}(x)\\
& \left(x-\omega(a)\right) a_n(x) =b_{n+1}(x)+\alpha_{2n+2}b_n(x)-\left(a-\beta_{2n+2}\right)P_{n+1}(x)+\gamma_{2n+1}P_{n}(x)
\end{align*}
which provides a suitable replacement of each term of identity \eqref{2orto-qd-2.7} when multiplied by $\left(x-\omega(a)\right)$, yielding:
\begin{align}
\nonumber&b_{n+1}(x)-\Big(x-\omega(a)+\left(a-\beta_{2n+1}\right)\left(a+p+\beta_{2n+1}\right) - \alpha_{2n+2}-\alpha_{2n+1}\Big) b_n(x)\\
\nonumber &+\Big(\alpha_{2n+1}\alpha_{2n} +\gamma_{2n} \left(p+\beta_{2n+1}+\beta_{2n-1}\right)\Big)b_{n-1}(x) + \gamma_{2n}\gamma_{2n-2}b_{n-2}(x) \\
\label{bs-ps} &+ \left( p+\beta_{2n+2}+\beta_{2n+1} \right) P_{n+1}(x)\\
\nonumber&+ \Big(\gamma_{2n+1}+\gamma_{2n}+\alpha_{2n+1} \left( p+\beta_{2n+1}+\beta_{2n} \right)  \Big)P_{n}(x)\\
\nonumber& +\Big( \alpha_{2n+1} \gamma_{2n-1}+\gamma_{2n}\alpha_{2n-1} \Big)P_{n-1}(x)=0\, ,\; n \geq 1.
\end{align}
Finally, when we multiply relation \eqref{2orto-qd-2.9} by $\left(x-\omega(a)\right)$ and we substitute all terms by an equal expression given by \eqref{2orto-qd-2.6} and \eqref{2orto-qd-2.8} we get:
\begin{align}
\nonumber&P_{n+2}(x)-\Big(x-\omega(a)+\left(a-\beta_{2n+2}\right)\left(a+p+\beta_{2n+2}\right) - \alpha_{2n+3}-\alpha_{2n+2} \Big) P_{n+1}(x)\\
\nonumber &+\Big(\alpha_{2n+2}\alpha_{2n+1} +\gamma_{2n+1}\left(p+\beta_{2n+2}+\beta_{2n}\right)\Big)P_{n}(x) +\gamma_{2n+1}\gamma_{2n-1}P_{n-1}(x)\\
 \label{ps-bs} &+ \left( p+\beta_{2n+3}+\beta_{2n+2} \right) b_{n+1}(x)\\
\nonumber&+ \Big(\gamma_{2n+2}+\gamma_{2n+1}+\alpha_{2n+2} \left( p+\beta_{2n+2}+\beta_{2n+1} \right)  \Big)b_{n}(x)\\
\nonumber& +\Big( \alpha_{2n+2} \gamma_{2n}+\gamma_{2n+1}\alpha_{2n} \Big)b_{n-1}(x)=0\, ,\; n \geq 1.
\end{align}

\medskip

The SCs defined in \eqref{main2MPS} guarantee that the coefficients of the three latest terms of identities \eqref{as-rs} - \eqref{ps-bs} vanish, because 
\begin{align*}
& p+\beta_{2n+3}+\beta_{2n+2}=0\, , \; \; n \geq 0\, ,\\
& p+\beta_{2n+2}+\beta_{2n+1}=0\, , \; \; n \geq 0\, ,\\
& \gamma_{2n+2}+\gamma_{2n+1}+\alpha_{2n+2} \left( p+\beta_{2n+2}+\beta_{2n+1} \right) = 0 \, , \; n \geq 0\, ,\\
&  \gamma_{2n+1}+\gamma_{2n}+\alpha_{2n+1} \left( p+\beta_{2n+1}+\beta_{2n} \right) = 0 \, , \; n \geq 1\, ,\\
&  \alpha_{2n+2} \gamma_{2n}+\gamma_{2n+1}\alpha_{2n}= 0 \, , \; n \geq 1\, ,\\
&  \alpha_{2n+1} \gamma_{2n-1}+\gamma_{2n}\alpha_{2n-1}= 0 \, , \; n \geq 1\, .
\end{align*}

This proves that the four polynomial sequences that arise from a general QD of the $2$-orthogonal MPS $\{ W_{n}\}_{n \geq 0}$ defined by \eqref{main2MPS} also fulfil recurrence relations of third order whose regularity is assured by the regularity of $\{ W_{n}\}_{n \geq 0}$. More precisely, we obtain for  n$ \geq 1$:
\begin{align}
\label{as-recrel} a_{n+1}(x)&=\Big(x-\omega(a)+\left(a-\beta \right)\left(a+p+\beta \right) - \alpha_{2}-\alpha_{1}\Big) a_n(x)\\
\nonumber &-\Big( \alpha_{2}\alpha_{1}+\gamma\left(p+2\beta \right) \Big)a_{n-1}(x) -\gamma^2 a_{n-2}(x)\, ,\; 
\end{align}
\begin{align}
\label{rs-recrel} R_{n+1}(x) & = \Big(x-\omega(a) + \left(a-\beta \right)\left(a+p+\beta \right)  -\alpha_{2}-\alpha_{1}\Big) R_n(x)\\
\nonumber &-\Big(\alpha_{2}\alpha_{1} + \gamma \left(p+2 \beta \right)\Big)R_{n-1}(x)-\gamma^2 R_{n-2}(x) \, ,\;
\end{align}
\begin{align}
\label{bs-recrel} b_{n+1}(x)&=\Big(x-\omega(a)+\left(a-\beta \right)\left(a+p+\beta \right) - \alpha_{2}-\alpha_{1}\Big) b_n(x)\\
\nonumber &-\Big(\alpha_{2}\alpha_{1} +\gamma \left(p+2\beta \right)\Big)b_{n-1}(x) - \gamma^2 b_{n-2}(x) \, ,\; 
\end{align}
\begin{align}
\label{ps-recrel} P_{n+2}(x)&=\Big(x-\omega(a)+\left(a-\beta \right)\left(a+p+\beta \right) - \alpha_{2}-\alpha_{1} \Big) P_{n+1}(x)\\
\nonumber &- \Big(\alpha_{2}\alpha_{1} +\gamma \left(p+2\beta \right)\Big)P_{n}(x) - \gamma^2 P_{n-1}(x)\, .\;
\end{align}
In brief, once we have the initial elements of each polynomial sequence we are able to fully characterise the general QD of $\{ W_{n}\}_{n \geq 0}$. We recall that these initial computations were already performed  with the help of the symbolic computational approach  and thus we have now proved the $2$-orthogonality of each polynomial sequence of the general QD. In particular, having obtained the data $a_{\nu}(x)=0\, , \, \nu=0,1, 2$ we easily conclude that $\{a_{n}\}_{n \geq 0}$ is trivial. In addition, we  remark that
\begin{align*}
\omega(a) - \left(a-\beta \right)\left(a+p+\beta \right)  +\alpha_{2}+\alpha_{1}=q+\alpha_{2}+\alpha_{1}+\left(p+\beta\right)\beta .
\end{align*}
When we consider the general QD of any $2$-orthogonal MPS $\{ \tilde{W}_{n}\}_{n \geq 0}$ defined as a perturbation of order $r$ of $\{ W_{n}\}_{n \geq 0}$, we also conclude that relations \eqref{as-recrel}-\eqref{ps-recrel} are fulfilled at least for $n \geq nmax$ for a certain positive integer $nmax > r$ which leaves the final conclusion about the $2$-orthogonality (and other aspects) to the specific data computed for $n = 0, \ldots nmax$.

\medskip

Finally, with respect to the classical character of each MPS obtained as a component sequence of the general QD and looking carefully to the data obtained for the MPS $\{ W_{n} (x)\}_{n \geq 0}$ solely defined by SCs  \eqref{main2MPS}, we must remark that the sequence $\{ R_{n} (x)\}_{n \geq 0}$ fulfils a recurrence relation of the following type, for given constants $\alpha$ and $\gamma$,
\begin{align*}
&\zeta_{n+1}(x)=x \zeta_{n}(x)-\alpha \zeta_{n-1}(x)-\gamma \zeta_{n-2}(x)\, , \quad n \geq 2,\\
& \zeta_{0}(x)=1,\quad \zeta_{1}(x)=x, \quad \zeta_{2}(x)=x^{2}-\alpha\, ,
\end{align*}
and therefore is $2$-classical as it is proved in  \cite{1997-KD-PM-1}, p.31. In fact, being $\{ \beta^{R}_{n} \}_{n\geq 0}$ a constant sequence, it is possible to assume $ \beta^{R}_{n}=0\, , \;n \geq 0$, after applying an affine transformation on $\{ R_{n}(x) \}_{n\geq 0}$ that preserves the $d$-orthogonality.
Moreover, we read also in that reference that the MPS $\{\zeta_{n}\}_{n \geq 0}$ is $2$-orthogonal and classical and the MPS $\{\zeta^{[1]}_{n}\}_{n \geq 0}$ fulfils
\begin{align*}
&\zeta^{[1]}_{n+1}(x)=x \zeta^{[1]}_{n}(x)-\alpha \frac{n(n+3)}{(n+1)(n+2)} \zeta^{[1]}_{n-1}(x)-\gamma \frac{(n-1)(n+4)}{(n+1)(n+2)} \zeta^{[1]}_{n-2}(x)\, , \quad n \geq 2,\\
& \zeta^{[1]}_{0}(x)=1,\quad \zeta^{[1]}_{1}(x)=x, \quad \zeta^{[1]}_{2}(x)=x^{2}-\frac{2}{3}\alpha.
\end{align*}

\section{Conclusions}

\noindent The symbolic implementation of the quadratic decomposition (QD) on a specific $2$-orthogonal MPS defined with constant structure coefficients, modulo two, allowed us to easily infer which of the four sequences of polynomials obtained through any QD are also $2$-orthogonal. The analytical proof of those characteristics was also developed. 

\noindent A co-recursive type $2$-orthogonal MPS was inserted in the symbolic procedure, along with further initial perturbations  on the structure coefficients. The output asserts that the main results gathered initially are maintained when those slight initial changes are performed. 
\noindent A further important characteristic was analysed strictly by computational means and the results proved that in some cases the obtained polynomial sequences are not classical in Hahn's sense. The situations where the classical character was detected can be connected to a concrete family studied in  \cite{1997-KD-PM-1}.

We should stress that the entire set of $2$-orthogonal MPSs quadratically decomposed are not $2$-classical as it was demonstrated symbolically, nevertheless, the components of the resultant general QD hold the $2$-orthogonality and either are $2$-classical or are perturbations of a $2$-classical MPS.

\section*{Acknowledgements}

\noindent  This work was partially supported by
\newline CMUP (UID/MAT/00144/2019), which is funded by FCT (Portugal) with national (ME) and European structural funds through the programs FEDER, under the partnership agreement PT2020.

\end{document}